\documentclass{amsart}
\usepackage[all]{xy}
\usepackage[T1]{fontenc}
\usepackage{times}
\usepackage{amssymb,epsfig,verbatim,xy}
\usepackage{enumitem}
\theoremstyle{plain}

\usepackage[T1]{fontenc}
\usepackage[utf8]{inputenc}
\usepackage[english]{babel}
\usepackage{amssymb}
\usepackage{amsthm}
\usepackage{graphics}
\usepackage{amsmath}
\usepackage{amstext}
\usepackage{multirow}
\usepackage{hyperref}

\usepackage{graphicx}
\usepackage{amsmath}
\usepackage{amssymb}
\usepackage{amsthm}
\usepackage{amstext}
\usepackage{epstopdf}
\usepackage{mathrsfs}
\usepackage{amscd}
\usepackage{color}

\newtheorem{thm}{Theorem}[section]
\newtheorem*{thm*}{Theorem}

\newtheorem{prop}[thm]{Proposition}
\newtheorem{cor}[thm]{Corollary}

\theoremstyle{remark}
\newtheorem{remark}[thm]{Remark}

\newcommand{\mb}{\mathbb}

\newcommand{\C}{\mb C}
\newcommand{\Pj}{\mb P}

\DeclareMathOperator{\im}{Im}

\numberwithin{equation}{section}

\addtocounter{section}{0}             
\numberwithin{equation}{section}

\author{Maycol Falla Luza} 
\address{Universidade Federal Fluminense, Brazil. }
\email{hfalla@id.uff.br}

\author{Frank Loray} 
\address{Univ Rennes, CNRS, IRMAR, UMR 6625, F-35000 Rennes, France. }
\email{frank.loray@univ-rennes.fr}

\author{Paulo Sad} 
\address{Instituto de Matematica Pura e Aplicada, Brazil. }
\email{sad@impa.br}

\begin{document}
\title{Non-algebraizable neighborhoods of curves}
\date{}
\begin{abstract}
We provide several families of compact complex curves embedded in smooth complex surfaces such that
no neighborhood of the curve can be embedded in an algebraic surface. Different constructions are proposed, 
by patching neighborhoods of curves in projective surfaces, and blowing down exceptional curves. 
These constructions generalize examples recently 
given by S. Lvovski. One of our non algebraic argument is based on an extension theorem
of S. Ivashkovich.
\end{abstract}

\maketitle
\tableofcontents

\section{Introduction}

This note is a complement to a recent nice paper \cite{Lvov} of S. Lvovski, where examples of positive embeddings of $\mb P^1$ in non algebraic surfaces are presented. A little earlier, a class of such examples appeared in \cite{FL2}: for each integer $m \geq 1$, an embedding of $C \simeq \mb P^1$ into some smooth surface is constructed, in such
a way as to have the self-intersection number $C\cdot C = m$, and the field of meromorphic functions of the surface is equal to $\mb C$; in particular, these surfaces are not algebraic. Then, in \cite{Lvov}, special embeddings of $\mb P^1$ in non algebraic surfaces are also constructed, but this time, the field of meromorphic functions has degree of transcendence $2$ over $\mb C$. We generalize here this construction by exhibiting examples of compact curves, poossibly singular, embedded in non-algebraizable surfaces with positive self-intersection numbers.

We use the same terminology as in  \cite{Lvov}: a smooth holomorphic surface $S$ is said {\it algebraizable} when there exists a holomorphic embedding of $S$ into a smooth projective surface $\hat{S}$; otherwise we say it is a non-algebraizable surface.

The paper is organized as follows. In Section \ref{Examples-in-positive}, we state our first main theorem, which describes how a neighborhood of a curve of the projective plane looses his algebraic character after a glueing with a neighborhood of a compact curve contained in some surface. Our main tool is a remarkable Levi’s extention type theorem due to S. Ivashkovich \cite{Iv}. 
Then, in Section \ref{Field-meromorphic}, we analyse several examples in order to understand the field of meromorphic functions of the resulting surfaces, we state here our main result: 

\begin{thm}\label{Ivas-trdeg2}
Let $C$ be a curve of degree $d\ge2$ in $\mb P^2$.
Then there exists a non algebraizable surface $S$ containing a copy of $C$
with $C\cdot C=d^2+k$ for any $k>0$. Moreover, the surface can be constructed in order to have field of meromorphic
functions $\C(S)$ with transcendence degree $0$, $1$ or $2$.
\end{thm}

In Section \ref{Lvov}, we generalize a construction of \cite{Lvov}. It arises naturally a phenomenon of non-separability of points by meromorphic functions, which seems to be a source of new examples not covered by the arguments of Section \ref{Examples-in-positive}. Finally, in Sections \ref{Neeman} and \ref{Examples-zero}, we construct other examples in the same spirit by starting with curves of zero self-intersection numbers contained in ${\mathbb P}^1$-bundles, and we finish  by comparing in Section \ref{comparing} this construction with the one presented in Section \ref{Examples-in-positive}

\section{Examples in the positive case}\label{Examples-in-positive}




Let $S$ be a smooth complex surface  with the following features:
\begin{enumerate}
\item $S$ contains two compact, connected, holomorphic curves $C$ and $D$ which cross each other transversely at only one point $P$, such that $D$ is smooth,
\item $C$ is biholomorphic to an algebraic curve $C_0 \subset \mathbb{P}^2$ and has a neighborhood $V\subset S$ biholomorphic to a tubular neighborhood $V_0$ of $C_0$ in $\mathbb{P}^2$; moreover, the biholomorphism from $V$ to $V_0$ takes $C$ to $C_0$.
\end{enumerate}

\begin{thm}\label{teorema-principal}
Let $S$, be a surface satisfying (1) and (2) as above.
Assume $\deg (C_0) \geq 2$. Then, the surface $S$ is not algebraizable.
\end{thm}

First observe that such a surface $S$ can be easily constructed by standard gluing process.
We first construct a surface $S$ using a standard glueing process. 
Let $V_0$ be a tubular neighborhood of the (connected) curve $C_0$ in $\Pj^2$. 
We also take a smooth compact curve $D_0$ contained in some surface, 
and consider again a tubular neighborhood $W_0$. We select points $p_0\in C_0$ and $q_0\in D_0$, and pick a local biholomorphism $\Psi:(V_0,p_0)\to (W_0,q_0)$ sending $C_0$ transversally to $D_0$. 
By glueing $V_0$ and $W_0$ by means of $\Psi$, we obtain a smooth complex surface $S$, 
which contains copies $(V,C)$ and $(W,D)$ of $(V_0,C_0)$ and $(W_0,D_0)$ respectively. 
The curve $C$ crosses $D$ transversally at a point $P$, which corresponds to the identification 
of $p_0$ with $q_0$ via $\Psi$.

\begin{figure}[h]
\centering
\includegraphics[scale=0.4]{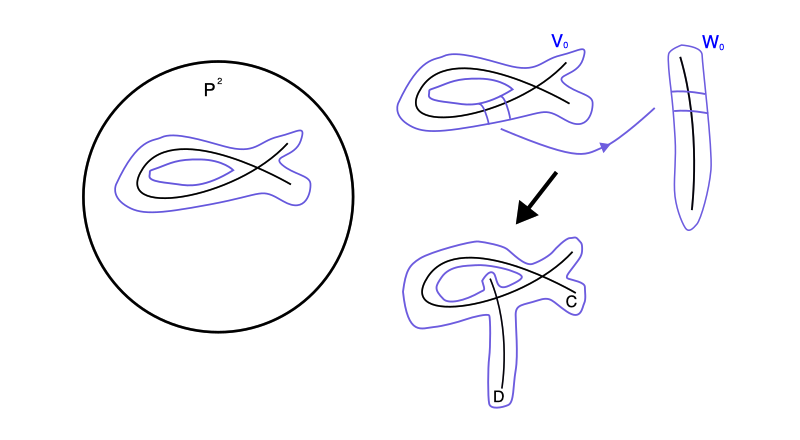}
\caption{Glueing Neighborhoods}
\label{figura:glueing}
\end{figure}

\begin{proof}
We proceed by contradiction. Assuming that  S is algebraizable, we may assume that $S$ is contained in some projective surface $\hat S$; furthermore, we have an holomorphic embedding   $i_0: V_0 \rightarrow S$ which takes $V_0$ to $V$ and $C_0$ to $C$.  Since $U_0:= \mathbb{P}^2\setminus C_0$ is a Stein surface, $K_0:=U_0 \setminus V_0$ has connected complement $V_0$, and $\hat S$ is projective, we may apply Ivashkovich's theorem \cite{Iv} and deduce that
the map  $i_0$ has  a meromorphic extension $I_0: \mathbb{P}^2 \dashrightarrow \hat{S}$ 
(mind that $i_0$ is already defined along $C_0$). 
Let us take in $V_0$ the component (branch) $b$ of $i_0^{-1}(D \cap V)$ which passes through $p_0=i_0^{-1}(P)$. Since $I_0^{-1}(D)$ is an algebraic curve of $\mathbb{P}^2$, it contains a component $B_0 \supset b$.

If $\deg C_0\geq 2$, then $B_0$ necessarily crosses $C_0$ at another point $p_0^{\prime}\neq p_0$, 
but $I_0|_{V_0}$ is a bijection between $V_0$ and $V$ which satisfies $I_0(p_0^{\prime}) \in B \cap C=\{P\}=I_0(0)$, a contradiction.
\end{proof}

As an application we have 

\begin{cor}\label{Cor:GenLvovski}
Let $C_0 \subset \Pj^2$ be an algebraic curve of degree $d \geq 2$ and let $k>d^2$ be a natural number. Then, there exist a non algebraizable surface $S$ and a holomorphic embedding $C$ of $C_0$ into $S$ with self-intersection $k$. 
\end{cor} 

\begin{proof} 
We repeat the construction of Theorem \ref{teorema-principal} with $D_0$ a $(-1)$-rational curve, i.e. a smooth rational curve
having self-intersection $(-1)$ in $W_0$. This provides a non-al\-ge\-brai\-za\-ble surface $S$ with a curve $C\cup D$ where $D$ is a $(-1)$-rational curve, $C$ is a copy of $C_0$ with self-intersection $C.C=d^2$, and $C.D=1$ (they intersect transversally at a single point $P$). 
Then, after blowing-down the curve $D$ to a point $P'$, we get a new surface $S'$ with a copy $C'$ of $C$ (or $C_0$) having now
self-intersection $C'.C'=d^2+1$. Obviously $S'$ cannot be contained in an algebraic surface $\hat S'$, otherwise
the blow-up $S\to S'$ would be contained in the algebraic surface obtained by blowing-up $P'$, contradiction. The pair $(S',C')$ is the pair $(S,C)$ of the statement for $k=d^2+1$. For greater $k$, we just repeat the operation of gluing $(-1)$-rational curves and contracting them, as many times as necessary. 
\end{proof}

\section{The field of meromorphic functions}\label{Field-meromorphic}
Let us consider the pair $(S,C)$ constructed in Corollary \ref{Cor:GenLvovski},
and let us denote by $\C(S)$ the field of global meromorphic functions on $S$. 
It obviously contains the field $\C$ as the subfield of constant functions. 
The transcendance degree of $\C(S)$ over $\C$ can be either $0$, $1$ or $2$ 
(it is bounded by the dimension of $S$). We show in this section that we can choose 
the gluing maps $\Psi$ in the construction of Theorem \ref{teorema-principal} and Corollary \ref{Cor:GenLvovski}
so as to choose the transcendance degree among $0$, $1$ or $2$.

First of all, in the proof of Corollary \ref{Cor:GenLvovski}, we note that $\C(S)=\C(S')$ since a meromorphic 
function is transformed into a meromorphic function under blow-up and blow-down.
On the other hand, a meromorphic function on the surface $S$ of Theorem \ref{teorema-principal}
is equivalent to the data of meromorphic functions $g:V_0\dashrightarrow\Pj^1$ and $h:W_0\dashrightarrow\Pj^1$
that coincide through the gluing map $\Psi:(V_0,p_0)\to (W_0,q_0)$, namely: $g=h\circ\Psi$. 

Now, we note that $\C(V)=\C(V_0)$ identifies with the field of rational functions $\C(\Pj^2)$
as any meromorphic function on the neighborhood $V_0$ of $C_0$ extends meromorphically on $\Pj^2$ (see \cite[Theorem 3.1]{Ros}).
The problem is now to understand which rational function $g$ on $V_0$ extends through $\Psi$
as a meromorphic function on $W_0$.

\subsection{Transcendance degree 2}
We can first choose $D_0$ as to be the exceptional divisor of the one-point-blow-up $\hat{\Pj}^2\to \Pj^2$
of the projective plane. Then, we can choose $\Psi$ to be birational map $\Psi: \Pj^2\dashrightarrow\hat{\Pj}^2$
that localizes as a local biholomorphism $(V_0,p_0)\to (W_0,q_0)$. Then, for any rational function $g$ on $\Pj^2$,
then $g\circ\Psi^{-1}$ is rational on $\hat{\Pj}^2$ and restricts as a meromorphic function on $W_0$.
We have just proved that $\C(\Pj^2)=\C(V)=\C(\underbrace{V\cup W}_{S})$. 
We can repeat this patching at several points of $C$ and get examples with transcendance degree
is $2$ with arbirtrary self-intersection number $k>d^2$.

\begin{remark}
When $d=2$ this is the construction in \cite[Section 4]{Lvov}.
\end{remark}

\subsection{Transcendance degree 1}
In order to give examples with lower transcendance degree, we can modify the 
previous construction as follows. First of all, consider a systems of affine coordinates $(x,y)$
in $\Pj^2$ such that $p_0$ is given by $(x,y)=(0,0)$ and $C_0$ is tangent to $y=0$ at $p_0$.
Then, consider usual affine charts on $\hat{\Pj}^2$ with coordinates $(r,t), (u,v)$ on $W_0$, 
related by $u=t^{-1}, v=rt$.
Again, we assume $q_0$ corresponds to $(r,t)=(0,0)$, and $D_0$ to $r=0$.
Then, we can define $\Psi$ in coordinates: $\Psi(x,y)=(r,t)$. The transversality condition 
is that $\Psi(x,0)$ must be transversal to $r=0$.

For $\Psi(x,y)=(x,y)$, or more generally $\Psi$ birational, we get transcendance degree
is $2$ as explained before. Now, let us consider $\Psi(x,y)=(xe^y,y)$. Then, a meromorphic function 
on $S$ is given by a rational function $g$ on $V_0$ such that the meromorporphic function 
$g\circ\Psi^{-1}=g(re^{-t},t)$ extends as a meromorphic function $h$ on $W_0$; in particular, 
in the other chart, $h(u,v):=g\left(uve^{-1/u},\frac{1}{u}\right)$ must be meromorphic at $(u,v)=(0,0)$,
which is possible if, and only if, $g$ does not depend on the first variable. Indeed, otherwise, the equality
$g\left(uve^{-1/u},\frac{1}{u}\right)=h(u,v)$ would say that $e^{-1/u}$ is solution of an implicit analytic
equation, impossible. Therefore, $\C(S)=\C(y)$ in that case, and has transcendance degree 1.

In order to produce examples with arbitrary self-intersection number $k>d^2$, one can patch 
(and contract) other $(-1)$-rational curves with birational gluing, so that we get no obstruction 
to further extend any $g(y)$ on neighborhood of other $(-1)$-rational curves.

\subsection{Transcendance degree 0}
Finally consider $(x,y)={\Psi}^{-1}(r,t)= (re^r,te^t)$. Then $h(u,v) = g(uv e^{uv},u^{-1}e^{u^{-1}})$ is meromorphic only when $g$ is constant. It follows that ${\mathbb C}(S)$ has degree of transcendence 0. This ends the proof of Theorem \ref{Ivas-trdeg2}.


\begin{remark}
The construction shows that, in Theorem \ref{Ivas-trdeg2}, smaller neighborhoods of $C$ are still non alge\-brai\-zable:
$S$ contains no algebraizable neighborhood of $C$.
We may therefore use the terminology {\it non algebraizable germ of surface containing $C$}.
\end{remark}


\section{The Lvovski Construction}\label{Lvov}
We present now a generalisation of the construction in \cite{Lvov}. 
Let $C_0\subset \mathbb P^2$ be a smooth plane curve of degree  $d=deg(C_0)\ge 2$ 
and take $D_0\subset \Pj^2$ a line transversal to $C_0$.  
The curves $C_0$ and $D_0$ cross each other in points $P_1,\cdots,P_d$.
We choose small tubular neighborhoods $V_0$ and $W_0$ of $C_0$ and $D_0$ respectively,
in such a way that their intersection $V_0\cap W_0$ splits into $d$ connected component $U_j\ni P_j$.
Denote $S_0:=V_0\cup W_0$.

We now consider copies $(V,C)$ and $(W,D)$ of $(V_0,C_0)$ and $(W_0,D_0)$, and 
denote by $p_j$ and $q_j$ the respective lifts of $P_j$. Denote also by $V_j\subset V$ and $W_j\subset W$ 
the respective copies of $U_j$, and $\Psi_j:V_j\to W_j$ defined by their identification to $U_j$.

We glue $V$ with $W$ through the map $\Psi_1:V_1\to W_1$, and denote by $S$ the surface
obtained by this way. 
It is equipped with an immersion $\pi:S\to \Pj^2$ coming from the two embeddings
$V\hookrightarrow\Pj^2$ and $W\hookrightarrow\Pj^2$: we have $\pi\circ\Psi_j=\pi$.
The two curves $C,D\subset S$ now intersect into a single point $P=\pi^{-1}(P_1)$
while $\pi^{-1}(P_j)=\{p_j,q_j\}$ for $j\ge2$.

\begin{figure}[h]
\centering
\includegraphics[scale=0.4]{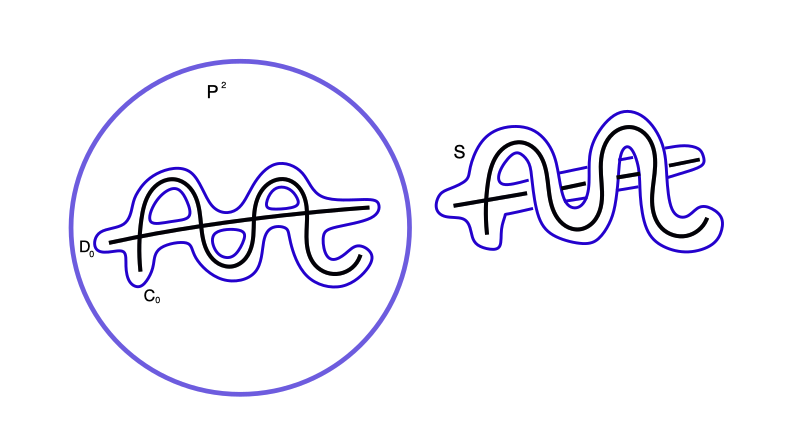}
\caption{Separation of points}
\label{figura:viaduto}
\end{figure}

\begin{prop}\label{Prop:Separation} The surface $S$ is not algebraizable.
\end{prop}

\begin{proof}
We first claim that the map $\tau:\C(\Pj^2)\stackrel{\sim}{\to}\C(S)$ given by
$f_0\mapsto f_0 \circ\pi$ defines an isomorphism. 
Indeed, given a meromorphic function $f\in\C(S)$, its restriction to $V$ is given by $f\vert_V=f_0\circ\pi$ for a meromorphic function $f_0\in \C(V_0)$.
But this latter one automatically extends as a rational function $\hat f_0$ on $\Pj^2$ (see \cite[Theorem 3.1]{Ros}). 
Then, the meromorphic function $\hat f_0\circ\pi$ on $S$ coincides with $f$ on $V$,
and therefore everywhere.

As a consequence, any meromorphic function $f\in\C(S)$ 
which is well-defined at $p_j$, will be also well-defined at $q_j$ and will take the same value
$f(p_j)=f(q_j):=f_0(P_j)$. On the other hand, rational functions separe points 
on projective surfaces. Indeed, assume by contradiction that $S$ is contained in a projective surface $\hat S$. 
Then, for any $j\ge2$, one might be able to find $\hat f\in\C(\hat S)$ such that 
$\hat f(p_j)\not=\hat f(q_j)$. By restricting $\hat f$ to $S$, we get a contradiction.
Therefore, $S$ is not algebraizable.
\end{proof}

Now let us consider two points $q_{d+1},q_{d+2}$ distinct from $q_1,\ldots,q_d$ on $D$,
and consider the blow-up $\hat S\to S$ of these points. Then, restricting to a tubular neighborhood
$S'=V\cup W'$ of the strict transform $C\cup D'$, we get an immersion $\pi':S'\to\Pj^2$;
the same argument as in the proof of Proposition \ref{Prop:Separation} shows that 
$\C(\Pj^2)\simeq\C(S')$ and $S'$ is not agebraizable for the same reason.

Finally, we observe that $D'$ is now a $(-1)$-rational curve that can be contracted: the blow-down
$(S',C\cup D')\to (S'',C)$ provides a non algebraizable surface with $\C(S'')\simeq\C(\Pj^2)$
that contains the curve $C$ with self-intersection $d^2+1$.
This provides an alternate contruction and proof for Theorem \ref{Ivas-trdeg2} in the case
of transcendance degree $2$. 

\begin {remark}
The method of taking such \'etale open set $\pi:S=V\sqcup_\Psi W\to S_0=V_0\cup W_0$
separating local branches of $C_0\cup D_0$ at $P_2,\ldots,P_d$ has already been considered
by two of the authors in \cite{FS}.
\end {remark}

\begin{remark}
In \cite{Lvov} the case $d=2$ is treated with a different argument to prove non algebraizability.
\end{remark}

When $d=1$, we can adapt and take $D$ as a smooth conic of $\mathbb P^2$, then blow-up 5 times and contract. 
This will provide non algebraizable neighborhoods of a smooth rational curve $C$ with self-intersection $C\cdot C\ge2$,
and improve a bit Theorem \ref{Ivas-trdeg2}. 
We can find in \cite[Section 5.3]{FL} another contruction for a rational curve $C$ with self-intersection $C\cdot C=1$.

\section{Neeman's Construction}\label{Neeman}

We recall here a construction by Neeman (\cite {N}) which generalises Serre's example, we will also use this reference for the definition of Ueda Type.

Let us consider a $\mathbb{P}^1$-bundle over a compact smooth curve $C$ of genus $g \geq 1$, whose changes of coordinates between trivializations $(z_j,u_j)$ are of the form
$$
u_i=\dfrac {u_j}{1+a_{ij}(z_j)u_j}
$$
We remark that $C$ is a section $(u_i=0)$ and its normal bundle is trivial, in particular $C\cdot C=0$. The collection $\{a_{ij}(z_j)\}$ is an element of $H^1(C,\mathcal O)$ and the bundle is trivial if this cocycle is zero in $H^1(C,\mathcal O)$.

When the functions $a_{ij}(z_j)=a_{ij}$ are constants, $\{a_{ij}\}\in H^1(C, \mathbb C)$, the foliation given by $du_i=0$ is well defined and $C$ is a leaf without singularities. 

We intend now to give examples where the complement of $C$ in the $\mathbb{P}^1$-bundle is a Stein surface. First of all, we need to find non trivial $\mathbb{P}^1$-bundles over $C$. Let us take the following short exact sequence of sheaves over $C$
$$
0 \rightarrow \mathbb C \rightarrow \mathcal{O} \rightarrow \Omega^1 \rightarrow 0
$$
which gives rise to the exact sequence in cohomology
$$
\cdots \rightarrow H^0(C,\Omega^1) \xrightarrow {\alpha} H^1(C,\mathbb C) \xrightarrow {\beta} H^1(C,\mathcal O) \rightarrow \cdots
$$
It is enough to show that $\beta$ is not identically zero. Now $\ker(\beta) = \im(\alpha)$; since $h^0(C, \Omega^1)=g$ and $h^1(C,\mathbb C)=2g$ (it is isomorphic to the first deRham cohomology group of $C$), it follows that $H^1(C,\mathbb C)$ is different from $\ker(\beta)$. 
For a constant cocycle $\{a_{ij}\}\in H^1(C, \mathbb C)$ which is not in the kernel $\ker(\beta)$, the corresponding 
$\mathbb{P}^1$-bundle is non trivial.

From now on, we consider a non trivial $\mathbb{P}^1$-bundle given by a constant cocycle which has $C$ as a section. Since the normal bundle is trivial, the Ueda type of $C$ is finite. Proceeding as in \cite[Proposition 5.3 and Remark 7.9]{N}, we conclude that the complement of $C$ is a Stein surface; the main point is that there exists a strictly plurisubharmonic function in a neighborhood of $C$ (see \cite{U}), which implies that $M \setminus C$ is holomorphically convex.

\section{Examples in the zero case}\label{Examples-zero}

We use now the examples of the last section to construct new non algebraizable surfaces. We proceed as follows:\\

\textbf{(1)} We take, as in the previous section, a $\mathbb{P}^1$-bundle $M$ over the curve $C$ such that $C$ is a section, $C\cdot C=0$ and $M \setminus C$ is a Stein surface and  take $V_1$ a tubular neighborhood of $C$. Let $D$ be a compact irreducible curve in $M$ which is transversal
to $C$ at some point  $p_1$.  Since $M$ is a compact projective surface, we may take $D$ containing some point $r\in C$ different from $p_1$.

\textbf{(2)} Let $B $ be a holomorphic curve contained in some surface $W$ and $W_1$ be a tubular neighborhood; we select
a point $q\in B$.

\textbf{(3)} We glue $V_1$ to $W_1$ using a local biholomorphism that sends $p_1$ to $q$ and sends a neighborhood of $p_1$ in $D$ into a neighborhood of $q$ in $B$.\\

The resulting surface $S$ is not algebraizable by a reason analogous to the one in Theorem 2.1. Suppose by contradiction that $S$ is contained in some compact projective surface $\hat S$. The holomorphic embedding $i$ of $V_1$ into $S$ extends to a meromorphic map $I$ from $M\setminus C$ to $\hat S$. Now $I^{-1}(B)$ is a holomorphic curve which contains $D$, so it must contain also $r$, which is a contradiction.  

As a consequence we have

\begin{thm}\label{Any-genus}
Let $C$ be a smooth compact curve of genus $g\geq 1$. Then there exists a non algebraizable surface $S$ containing a copy of $C$
with $C\cdot C=k$ for any $k>0$.
\end{thm}

\begin{proof}
We may use $B$ as a rational $(-1)$ curve, and replace $k-1$ fibers by $(-1)$ rational curves. We get a non algebraizable surface, and after $k$ blow-down's we arrive to a non algebraizable surface containing $C$ as a curve satisfting $C\cdot C=k$ (in fact, a non algebraizable germ of surface containing $C$).
\end{proof}

\section{Comparing non algebraizable neighborhoods}\label{comparing}

Let us compare two non algebraizable neighborhoods of curves built from the  processes we decribed before.

We start taking a neighborhood in $\mathbb P^2$ of a smooth plane curve $C$ of degree $d\ge 3$ (therefore its genus is $\ge1$). As in Section \ref{Examples-in-positive} we  glue a $(-1)$ rational curve at a point $Q \in C$ and blow it down in order to get a surface $S_1$ containing a copy $C_1$ of $C$ whose self-intersection number is $l=d^2+1$.\\

The next step comes from Section \ref{Examples-zero}: we take a non trivial $\mathbb P^1$-bundle over $C$ with a foliation around $C$ given by $\{du_i=0\}$ (containing $C$ as a leaf) and glue $l$ (-1) rational curves at points $P_1,\dots P_l$. We may take foliations around the $(-1)$ curves which become radial singularities after blow-down`s; moreover, these foliations glue with the foliation given by $\{du_i=0\}$ in the ${\mathbb P^1}$-bundle. After blowing down these curves we have another copy $C_2$ of $C$ with self-intersection number $l$ inside a surface $S_2$. We have also a foliation $\mathcal F_2$ on $S_2$ for which $C$ is an invariant curve with $l$ radial singularities.

\begin{prop} 
Surfaces $S_1$ and $S_2$ are not biholomorphically equivalent. 
\end{prop}
\begin{proof}
Suppose that $\phi: S_2\rightarrow S_1$ is a biholomorphism; we call $\mathcal F_1=\phi(\mathcal F_2)$ and $Q_j=\phi (P_j)$. Two possibilities arise.

\noindent {\bf (a)}  $Q\notin \{Q_1\dots Q_l\}$. The blow-up of ${\mathcal F}_1$ at $Q$ leads to a foliation $\tilde{\mathcal F}_1$ defined in a neighborhood of $C\subset \mathbb P^2$ (and a fortiori in $\mathbb P^2$) with $l$ radial singularities above $Q_1\dots Q_l$ and an extra singularity which is a zero of order 1 of $\tilde {\mathcal F}_1$ along $C$ (locally over $Q$ we have a local expression $tx=c$ for $\tilde {\mathcal F}_1$).

It follows from Brunella's formula (\cite[Chapter 2, Proposition 3]{Br})

$$
(deg(\tilde {\mathcal F}_1)+2).d= d^2+1+d^2
$$

\noindent which is a contradiction.

\noindent {\bf (b)} we suppose now that $Q\in \{Q_1,\dots, Q_l\}$, say $Q=Q_1$. The foliation $\tilde {\mathcal F}_1$ has $C$ as a invariant curve that contains $d^2$ radial singularities (the blow-up at $Q_1$  produces a regular foliation transverse to the (-1) rational curve). In homogeneous coordinates  $\tilde {\mathcal F}_1$ is defined by a 1-form $\Omega$ and $C$ by a homogeneous polynomial $F=0$ related by
$$
\Omega=GdF+F\eta
$$
\noindent where $\eta$ is also a holomorphic 1-form and $G$ is a homogeneous polynomial of the same degree as $F$. Clearly $G=0$ exactly at the singularities above $Q_2,\dots, Q_l$. The level curves of $\dfrac{G}{F}$, after blow-down, become a family of compact curves in $S_1$ passing transversely through all the points $Q_1,\dots, Q_l$. This family can be transported to $S_2$ and blown-up to the $\mathbb P^1$-bundle, producing an infinite number of sections 2 by 2 disjoints; this is a contradiction since the bundle is not trivial.
\end{proof}

We proceed now to compare two non algebraizable surfaces constructed as in Section \ref{Examples-in-positive} when $C$ is smooth. Let us repeat the construction: we have a plane algebraic curve $C$ of degree greater or equal to 2. At some point $Q \in C$ we glue a rational (-1) curve $D$ transverse to $C_0$ getting a surface $\hat S$. The surface $S$ we are interested in is the blow down of $D$ to a point $Q(D)$. Suppose now that we have another surface  $S^{\prime}$ obtained from the same procedure: we glue a (-1) rational curve $D^{\prime}$ to $C_0$ at a point $Q^{\prime}$  in order to get a surface ${\hat S}^{\prime}$ and blow down $D^{\prime}$ to a point $Q^{\prime}(D^{\prime})\in S^{\prime}$.

\begin{prop}  Any biholomorphic equivalence between $S$ and $S^{\prime}$ that preserves $C$ takes $Q(D)$ to $Q^{\prime}(D^{\prime})$.
\end{prop}

\begin {proof} 
We take a fibration over $C$ in a neighborhood of this curve in $\hat S$; evidently there exists a finite number of points of tangency between the 
fibration and $C$ (we may assume that the tangency order is 1 at those points).We may also suppose that the fiber through $Q$ is transverse to $D$ and $C$. The existence of this fibration follows, for example, by taking a linear pencil in ${\mathbb P}^2$ with base point not in $C$. We look now the images of the fibers in $S$; there is a special feature as an infinite number of them pass thouroug $Q(D)$.

A biholomorphis between $S$ and $S^{\prime}$ can be lifted to a biholomorphism between neigh\-bor\-hoods of $C$ in $\hat S$ and ${\hat S}^{\prime}$, and takes a fibration as above over $C$ in $\hat S$ to a fibration over $C$ in ${\hat S}^{\prime}$ with the same properties (in fact, we need that the fiber over $Q^{\prime}$ is also transverse to $C$ and $D^{\prime}$, but this can be arranged by changing the pencil).

Finally, since $Q(D)$ and $Q^{\prime}(D^{\prime})$ are the only points contained in an infinite number of images of the fibers, we conclude that one is sent to the another one by the biholomorphism.
\end{proof}

As an application, in the case $\deg C \geq 3$ we see that there exists a $1$-parameter family of different non alge\-brai\-za\-ble surfaces (up to biholomophisms). It is enough to observe that there exists only a finite number of possible images of $Q(D)$ under the group of automorphisms of $C$ (this group is finite by Hurwitz's Theorem).\\

The above discussion applies to the non algebraizable surface that appear in Section 5 (this time the fibrations have no tangency points). The result is the same as before.

\thanks{ {\bf Acknowledgements:} 
The authors thank Brazilian-French Network in Mathematics and
CAPES/COFECUB Project Ma 932/19 “Feuilletages holomorphes et intéractions avec la géométrie”. 
Falla Luza acknowledges support from CNPq (Grant number 402936/2021-3).
Loray is supported by CNRS and  The Centre Henri Lebesgue, program ANR-11-LABX-0020-0.}

\end{document}